\documentclass[12pt,a4paper]{amsart}

\usepackage[english]{babel}
\usepackage[T1]{fontenc}
\usepackage[utf8]{inputenc}
\usepackage{lmodern}
\usepackage{amsmath,amsthm,amscd,amsfonts,amssymb}
\usepackage{mathtools}
\usepackage{microtype}
\usepackage[shortlabels]{enumitem}
\usepackage{needspace}
\usepackage{hyperref}
\usepackage[capitalise]{cleveref}
\usepackage{setspace}
\allowdisplaybreaks

\newtheorem{theorem}{Theorem}[section]
\newtheorem{lemma}[theorem]{Lemma}
\newtheorem{proposition}[theorem]{Proposition}
\newtheorem{corollary}[theorem]{Corollary}

\theoremstyle{definition}
\newtheorem{definition}[theorem]{Definition}

\theoremstyle{remark}
\newtheorem{remark}[theorem]{Remark}

\newcommand{\GL}{\mathbf{GL}}
\newcommand{\On}{\mathrm{On}}
\newcommand{\Reg}{\operatorname{Reg}}
\newcommand{\Sing}{\operatorname{Sing}}
\newcommand{\NS}{\mathrm{NS}}
\newcommand{\Tr}{\operatorname{Tr}}
\newcommand{\Logd}{\operatorname{Log}_{d}}
\newcommand{\Pow}{\mathcal P}
\newcommand{\calI}{\mathcal I}
\newcommand{\calF}{\mathcal F}

\newcommand{\equivNS}{=_{\NS}}

\newcommand{\vecU}{\vec U}

\providecommand{\GLP}{\mathbf{GLP}}
\providecommand{\Logp}{\operatorname{Log}_{d}}
\providecommand{\llbracket}{[\![}
\providecommand{\rrbracket}{]\!]}

\title[GL and finite derived ordinal topologies]{Completeness of the G\"odel--L\"ob Provability Logic for Finite Derived Ordinal Topologies}

\author[M. Golshani]{Mohammad Golshani}
\address{School of Mathematics, Institute for Research in Fundamental Sciences (IPM), P.O. Box 19395-5746, Tehran, Iran}
\email{golshani.m@gmail.com}


\begin{document}

\begin{abstract}
We study the G\"odel--L\"ob provability logic \(\GL\) for finite derived topologies on the ordinals.
 Assuming GCH and the existence of a measurable cardinal $\kappa$ which carries a coherent sequence of normal measures $\vecU$ with
      $o^{\vecU}(\kappa)\geq\kappa$, we obtain a cofinality and GCH preserving generic extension in which
$\Logd(\kappa,\tau_n)=\GL$ for every $n<\omega$, where $\tau_n$ is the $n$-th derived topology. The result answers a well-know open question.
Extending Beklemishev's theorem,  we show that in the same model, for every $k<\omega$,
$
 \Logp(\kappa;\tau_{2k},\tau_{2k+1})=\GLP_2,
$
where the two modalities of $\GLP_2$ are interpreted respectively by
$d_{2k}$ and $d_{2k+1}$.
\end{abstract}
\maketitle

\section{Introduction}

The G\"odel--L\"ob provability logic \(\GL\) is the modal logic generated by the normal modal axiom together with L\"ob's axiom.  By the classical completeness theorem of Segerberg, it is complete for finite transitive irreflexive trees.  Esakia observed that the modal derivative on scattered topological spaces has the same formal behavior, and this opened the way to ordinal and set-theoretic semantics for provability logic.

Blass~\cite{Blass} introduced a particularly useful filter semantics on the ordinals.  His construction reduces completeness to the existence of suitable disjoint labelings of a family of universal trees \(K_n\).  For the end-segment filters these labelings exist in ZFC, giving the Abashidze--Blass ordinal completeness theorem.  For the club filters Blass obtained completeness from suitable square principles.  The same labeling method was later used for the filter sequence of normal measures in~\cite{GolshaniZoghifard}.

One can define a hierarchy \(\langle\tau_n:n<\omega\rangle\) of derived topologies on the ordinals.  We follow Bagaria's convention that \(\tau_0\) is the interval topology; then \(\tau_1\) is the club topology and \(\tau_2\) is the Mahlo topology. The completeness problem beyond the club topology has remained open since the work of Blass \cite{Blass}, see also
\cite{AguileraFernandezDuque}, and Bagaria
\cite{Bagaria}.

We write \(\Logd(\alpha,\tau_n)\) for the set of modal formulas
valid on the ordinal space \((\alpha,\tau_n)\) under derivative semantics.
Since each \(\tau_n\) is scattered, Esakia's theorem always gives the
soundness inclusion
\(\GL\subseteq\Logd(\alpha,\tau_n)\).  The problem is to obtain the reverse
inclusion.
In this paper we answer this question and prove something stronger.
The main result of the paper is the following.

\begin{theorem}\label{thm:intro-main}
Assume GCH.  Suppose that \(\kappa\) carries a coherent sequence of normal measures \(\vec U\) with \(o^{\vec U}(\kappa)\geq\kappa\).  Then there is a cofinality and GCH preserving forcing extension, such that
\[
  \Logd(\kappa,\tau_n)=\GL\qquad\text{for every }n<\omega.
\]
\end{theorem}
Furthermore, we show that our approach show that in the same generic extension,  for every $n<\omega$,
\[
 \Logp(\kappa;\tau_{2n},\tau_{2n+1})=\GLP_2,
\]
where the two modalities of $\GLP_2$ are interpreted respectively by
$d_{2n}$ and $d_{2n+1}$.

\section{Preliminaries}
\label{sec:prelim}
\subsection{The provability logic $\GL$}

We briefly recall the G\"odel--L\"ob provability logic $\GL$. It is the normal modal logic
obtained from the propositional tautologies by adding the axioms
\[
   \Box(\varphi\rightarrow\psi)
   \rightarrow
   (\Box\varphi\rightarrow\Box\psi)
\]
and L\"ob's axiom
\[
   \Box(\Box\varphi\rightarrow\varphi)\rightarrow\Box\varphi,
\]
and closing under modus ponens and necessitation
\[
   \frac{\varphi}{\Box\varphi}.
\]
Thus $\GL$ is the logic of finite transitive irreflexive
Kripke frames.
We have the following result of Esakia.
\begin{theorem}(\cite{Esakia})
A modal formula $\varphi$ belongs to $\GL$ if and only if
$\varphi$ is valid, under Cantor-derivative semantics, in every
scattered topological space.
\end{theorem}

\subsection{Filter semantics}

For an ordinal $\alpha$ and a filter $\calF\subseteq\Pow(\alpha)$, define its positive sets by
\[
  \calF^+=\{A\subseteq\alpha:(\forall X\in\calF)\ A\cap X\neq\emptyset\}.
\]
We allow the filters to be improper, in which case $\calF^+=\emptyset$.
A filter sequence below an ordinal $\Omega$ is a family
${\calF}=\langle\calF_\alpha:\alpha<\Omega\rangle$, where each $\calF_\alpha$ is a
filter on $\alpha$.  We assign to ${\calF}$, its associated derivative operation by
\[
  d_{\calF}(A)=\{\alpha<\Omega:A\cap\alpha\in\calF_\alpha^+\}.
\]

 For
$n<\omega$, let $K_n$ be the tree whose nodes are the empty set $\emptyset$ and all
finite sequences
$
  s=\langle(i_1,j_1),\ldots,(i_{k},j_{k})\rangle
$
such that
$
  n>i_1>\cdots>i_{k}\geq0,$
and
$ j_1,\ldots,j_{k}<\omega.
$
The order relation is the end extension relation $\lhd$.

\begin{lemma}\cite{Blass}
\label{lem:treeuniversal}
Every finite rooted transitive irreflexive tree of height $n$ is a bounded
morphic image of $K_n$.
\end{lemma}
\begin{theorem}\cite{Blass}
\label{thm:blasscriterion}
Let $\calF=\langle\calF_\alpha:\alpha<\Omega\rangle$ be a filter sequence.  Suppose
that for every $n<\omega$ there is a map
$\Gamma:K_n\to\Pow(\Omega)$ such that:
\begin{enumerate}[(1)]
\item $\Gamma({\emptyset})$ is non-empty;
\item the sets $\Gamma(s)$, $s\in K_n$, are pairwise disjoint;
\item if $s\triangleleft t$ and $\alpha\in\Gamma(s)$, then
      $\Gamma(t)\cap\alpha\in\calF_\alpha^+$;
\item if $\alpha\in\Gamma(s)$, then
   $
        \bigcup_{s\triangleleft t}(\Gamma(t)\cap\alpha)\in\calF_\alpha.
$
\end{enumerate}
Then every formula valid for the filter sequence is provable in $\GL$.
\end{theorem}
We now introduce a definition which is implicit in the work of Blass \cite{Blass}.
\begin{definition}
\label{def:rooted}
Suppose $\calF$ and $K_n$ are as above. A family
$\langle B_s:s\in K_n\rangle$  is an  $(\calF, \rho)$-realization if:
\begin{itemize}
\item $B_{\emptyset}=\{\rho\}$;
\item $B_s\subseteq\rho$ for $s\neq \emptyset$;
\item the sets $B_s$ are pairwise
disjoint;
\item If $ s\triangleleft t$ and $\eta \in B_s$, then
  $B_t\cap\eta\in  \calF_\eta^+$;
\item If $\eta \in B_s$, then  $\bigcup_{s\triangleleft t}(B_t\cap\eta)\in\calF_\eta.$
\end{itemize}
\end{definition}
The following lemma is evident.
\begin{lemma}
\label{lem:rootrestriction}
Suppose $\Gamma$ satisfies Theorem \ref{thm:blasscriterion} and
$\rho\in\Gamma(\emptyset)$, then $\langle B_s:s\in K_n\rangle$
is an  $(\calF, \rho)$-realization, where
$B_{\emptyset}=\{\rho\}$ and
$B_s=\Gamma(s)\cap\rho$, for
$s\neq \emptyset.
$
\end{lemma}
We need the following two results, see \cite{Blass}

\begin{theorem}[Abashidze--Blass]
\label{thm:tailbase}
In ZFC, $\GL$ is complete for the end-segment filter sequence. Indeed, for every
$n<\omega$, there $0<\rho<\omega^\omega$ and  an  $(\calF, \rho)$-realization, where $\calF=\langle \calF_\alpha: \alpha < \omega^\omega    \rangle$
is end-segment filter sequence.
\end{theorem}

\begin{theorem}[Blass]
\label{thm:clubbase}
Assume Jensen's $\square_\mu$ for every infinite cardinal $\mu<\aleph_\omega$.
Then $\GL$ is complete for the club-filter sequence. Indeed, for every
$n<\omega$, there is $0<\rho<\aleph_\omega$ and  an  $(\calF, \rho)$-realization, where $\calF=\langle \calF_\alpha: \alpha < \aleph_\omega    \rangle$
is the club filter sequence.
\end{theorem}

\subsection{The  derived topology}
 If $\tau$ is a topology on the ordinals and $A$ is a class of
ordinals, let
\[
d_\tau(A)=
\{\alpha:\alpha\text{ is a limit point of }A\text{ with respect to }\tau\}
\]
denote the Cantor derivative of $A$. In what follows, we follow Bagaria's notation from \cite{Bagaria}, and use
$\tau_0$ for the interval topology. We define
the hierarchy $\langle\tau_n:n<\omega\rangle$, by induction, as follows:
\begin{itemize}
\item $\tau_0$ is the interval topology on ordinals,
\item  $\tau_{n+1}$ is the topology generated by
$\tau_n$ all its derived sets $d_{n}(A)$, where  $d_n=d_{\tau_n}$.
\end{itemize}
Note that by our convention,
$\tau_1$ is the \emph{club topology}, and $\tau_2$ is the
\emph{Mahlo topology}.

\subsection{Polymodal provability logic}

For $r\geq1$, let $\GLP_r$ denote Japaridze's polymodal provability logic in
the language with modalities $[0],\ldots,[r-1]$.  In addition to the axioms of $\GL$, it contains, for $i<j<r$, the  axioms
$
 [i]\varphi\longrightarrow[j]\varphi,
$
and
$
 \langle i\rangle\varphi\longrightarrow[j]\langle i\rangle\varphi.
$
If $(X,\theta_0,\ldots,\theta_{r-1})$ is a polytopological space, a
valuation $v$ is extended to all formulas by Boolean recursion and
$
 \llbracket\langle i\rangle\varphi\rrbracket_v
 =d_{\theta_i}(\llbracket\varphi\rrbracket_v).
$
Let
$
 \Logp(X;\theta_0,\ldots,\theta_{r-1})
$
denote the set of formulas which are valid under every such valuation.
Note that every increasing sequence $i_0<\cdots<i_{r-1} < \omega$, 
$
 \GLP_r\subseteq
 \Logp(\Omega;\tau_{i_0},\ldots,\tau_{i_{r-1}}). 
$
We have the following result.
\begin{theorem} (\cite{BeklemishevGLB})
\label{beklem}
It is consistent that  $\GLP_2=
 \Logp(\Omega;\tau_{0},\tau_{1}),$ where $\Omega \geq \aleph_\omega.$
\end{theorem} 
Indeed the above result holds under the assumption of axiom of constructibility, or if we just assume Jensen's $\square_\mu$ holds for every infinite cardinal $\mu<\aleph_\omega$.

\subsection{Simultaneous-stationarity hierarchy}
In this subsection we introduce some results from Bagaria \cite{Bagaria}.

\begin{definition}
\label{def:sstat}
Let $A\subseteq\alpha$.
\begin{enumerate}[(1)]
\item $A$ is \emph{$0$-s-stationary}\footnote{``s-stationary'' is an abbreviation of ``simultaneously stationary''.} in $\alpha$ if it is
      unbounded in $\alpha$.
\item For $0<n<\omega$, $A$ is \emph{$n$-s-stationary} in
      $\alpha$ if, for every $k<n$ and every pair $S,T\subseteq\alpha$ which
      are $k$-s-stationary in $\alpha$, there is $\beta\in A$
      such that both $S\cap\beta$ and $T\cap\beta$ are
      $k$-s-stationary in $\beta$.
\item The ordinal
$\alpha$ is $n$-s-reflecting if $\alpha$ is
$n$-s-stationary in $\alpha$.
\end{enumerate}

\end{definition}

For every ordinal $\alpha$ and $n<\omega$, let
\begin{align*}
  \calI^n_\alpha
  &=\{A\subseteq\alpha:A\text{ is not }n\text{-s-stationary in }\alpha\},\\
  \calF^n_\alpha
  &=(\calI^n_\alpha)^*
    =\{X\subseteq\alpha:\alpha\setminus X\in\calI^n_\alpha\}.
\end{align*}
\begin{theorem}\cite{Bagaria}
\label{thm:bagaria}
For every $n < \omega,$ the following statements hold.
\begin{enumerate}[(1)]
\item $A$ is $1$-s-stationarity iff $A$ is stationary.
\item For every $A\subseteq\On$,
      \[
        d_n(A)=\{\alpha:A\cap\alpha\text{ is }n\text{-s-stationary in }\alpha\}.
      \]
\item $A$ is $(n+1)$-s-stationary in $\alpha$ iff, for every $k\leq n$ and
      every $k$-s-stationary sets $S,T\subseteq\alpha$, we have
$
A\cap d_k(S)\cap d_k(T)\cap\alpha\neq \emptyset.
$
\item If $A$ is $n$-s-stationary in $\alpha$ and $C\subseteq\alpha$ is club,
      then $A\cap C$ is $n$-s-stationary in $\alpha$.
\item $\calI^n_\alpha$ is a proper ideal iff $\alpha$ is $n$-s-reflecting.
      In that case $\calF^n_\alpha$ is  also a proper filter and
      \[
        (\calF^n_\alpha)^+
        =\{A\subseteq\alpha:A\text{ is }n\text{-s-stationary in }\alpha\}.
      \]
\item If $n>0$, the non-stationary ideal $\NS_\alpha$ is contained in $\calI^n_\alpha$.
\item $X\in\calF^n_\alpha$ iff, for some $k<n$ and some
$k$-s-stationary set $S,T\subseteq\alpha$, we have
$
  d_k(S)\cap d_k(T)\cap\alpha\subseteq X.
$
\end{enumerate}
\end{theorem}
the following easy lemma show that the above notions are invariant modulo the non-stationary ideal.
\begin{lemma}
\label{lem:NSinvariance}
Suppose $1\leq n<\omega$,  $\alpha$ has uncountable cofinality, and
$A\mathbin\triangle B\in\NS_\alpha$.  Then
\begin{enumerate}[(i)]
\item $A$ is $n$-s-stationary iff $B$ is $n$-s-stationary;
\item $A\in\calF^n_\alpha$ iff $B\in\calF^n_\alpha$;
\item $A\in(\calF^n_\alpha)^+$ iff $B\in(\calF^n_\alpha)^+$.
\end{enumerate}
\end{lemma}

\begin{proof}
Let $C$ be a club of $\alpha$ such that $A \cap C=B \cap C.$

(i).
 If $A$ is
$n$-s-stationary, then by
Theorem \ref{thm:bagaria}(4), $A\cap C=B\cap C$ is $n$-s-stationary and hence so is $B$. By symmetry, the converse holds as well.

(ii). It follows by applying  (i) to the sets $\alpha \setminus A$ and $\alpha \setminus B$.

(iii).  If the
filter is proper, (iii) follows from Theorem \ref{thm:bagaria}(5); if it is improper,
both positive families are empty.
\end{proof}

\subsection{Soundness}

Let $\delta$ be a limit ordinal and $\alpha<\delta$. Set
\[
 \mathcal N^n_\alpha=
 \{X\subseteq\alpha:\text{for some }U\in\tau_n,
 \ \alpha\in U\text{ and }U\cap\alpha\subseteq X\}.
\]

\begin{proposition}
\label{prop:neighborhood}
For every $n<\omega$, $\mathcal N^n_\alpha=\calF^n_\alpha$.
\end{proposition}

\begin{proof}
This is clear for $n=0$, thus suppose that $n>0$. We may assume that $\calF^n_\alpha$ is proper, as otherwise the conclusion is clear.

First suppose that $X \in \calF^n_\alpha.$ Then   $\alpha \setminus X \in \calI^n_\alpha$, thus, by Theorem \ref{thm:bagaria}(2), there are $k<n$ and
$k$-s-stationary set $S,T\subseteq\alpha$ such that  $(\alpha \setminus X) \cap  d_k(S)\cap d_k(T)=\emptyset.$
Note that $\alpha \in d_k(S)\cap d_k(T) \in \tau_n$  and
$d_k(S)\cap d_k(T)\cap\alpha \subseteq X.$
By its definition, $d_k(S)\cap d_k(T)\cap\alpha \in \mathcal N^n_\alpha$ and hence $X \in \mathcal N^n_\alpha.$

Conversely, suppose $X \in \mathcal N^n_\alpha.$ Take $U\in\tau_n$ with $\alpha\in U$ such that $U\cap\alpha\subseteq X$.  Let $V\subseteq U$ be a basic open neighborhood,
hence,
\[
  V=I\cap\bigcap_{i<r}d_{k_i}(A_i),
 \qquad k_i<n,
\]
where $I$ is open in the interval topology and contains $\alpha$, and the sets
$A_i\cap\alpha$ are $k_i$-s-stationary. But then
$
 d_{k_i}(A_i)\cap\alpha=d_{k_i}(A_i\cap\alpha)\cap\alpha
 \in\calF^n_\alpha,
$
and $I\cap\alpha \in \calF^0_\alpha\subseteq\calF^n_\alpha$.  It follows that $V\cap\alpha\in \calF^n_\alpha$ and $V \cap \alpha \subseteq U \cap \alpha \subseteq X$. Thus,
$X\in\calF^n_\alpha$.
\end{proof}

Every $\tau_n$ refines the scattered interval topology, so every
$(\delta,\tau_n)$ is scattered. Thus by the work of Esakia \cite{Esakia}, and Proposition \ref{prop:neighborhood}, we have the following soundness result.

\begin{corollary}[Soundness]
\label{cor:soundness}
For every limit ordinal $\delta$ and every $n<\omega$,
\[
  \GL\subseteq\Logd(\delta,\tau_n).
\]
Furthermore, filter validity for $\langle\calF^n_\alpha:\alpha<\delta\rangle$ is the same
as derivative validity on $(\delta,\tau_n)$.
\end{corollary}

\section{Reflection and coherent systems}
\label{sec:forcing}

In this section, we recall some results from \cite{Witzany}, and prove some further results which are needed for the later sections.
Let $\lambda$ be a regular uncountable cardinal.  Let $\Reg(\lambda)$ denote the set of regular
cardinals below $\lambda$ and $\Sing(\lambda)$ be its complement modulo a
bounded initial segment.  For a stationary  set $S\subseteq\lambda$, let
\[
  \Tr_\lambda(S)=\{\beta<\lambda:S\cap\beta
  \text{ is stationary in }\beta\}.
\]
Also, for stationary set $S,T\subseteq\lambda$, let us define
\[
  S<T\quad\Longleftrightarrow\quad
  T\setminus\Tr_\lambda(S)\in\NS_\lambda.
\]

\begin{definition}
\label{def:witzany-system}
\begin{enumerate}
\item A \emph{system of normal measures} up to a regular uncountable cardinal
$\kappa$ is a sequence
$
   \mathcal S=\langle\mathcal S_\alpha:\alpha\leq\kappa\rangle,
$
where every member of $\mathcal S_\alpha$ is a normal measure on
$\alpha$.
\item The system $\mathcal S$ is \emph{closed} if, whenever
$\alpha\leq\kappa$, $W\in\mathcal S_\alpha$, and
$
   j_W:V\longrightarrow M_W=\operatorname{Ult}(V,W)
$
is the corresponding ultrapower embedding, then
$
   (j_W\mathcal S)(\alpha)\subseteq\mathcal S_\alpha.
$

\item For $U,W\in\mathcal S_\alpha$ put
$
   U\mathrel{\triangleleft_{\mathcal S}}W
  \Longleftrightarrow
   U\in(j_W\mathcal S)(\alpha).
$
The relation $\triangleleft_{\mathcal S}$ is transitive and
well-founded.

\item The measures in $\mathcal S_\alpha$ are called \emph{separable} if
there is a family
$
   \langle X_U^\alpha:U\in\mathcal S_\alpha\rangle
$
of subsets of $\alpha$ such that, for all
$U,W\in\mathcal S_\alpha$,
$X_U^\alpha\in W$
   iff
  $ U=W.$
Such a family will be called a \emph{separating family}.

\item We call
$\mathcal S$ a \emph{closed system of separable measures} if it is
closed and each
$\mathcal S_\alpha$ is separable.
\end{enumerate}
\end{definition}

\begin{definition}
Suppose $P$ is a well-founded poset. The reflection
ordering \emph{realizes} $P$ if there is a maximal antichain
$
   \langle X_p:p\in P\rangle
$
of stationary subsets of $\Reg(\lambda)$ such that, for all
$p,q\in P$ and all stationary sets
$S\subseteq X_p$ and $T\subseteq X_q$,
\[
   S<T
   \quad\Longleftrightarrow\quad
   p<_Pq.
\]
\end{definition}

We use the following theorem of Witzany in an essential way.

\begin{theorem}(\cite[Theorem~1]{Witzany})
\label{thm:witzany-general}
Assume GCH, and let
$
   \mathcal S
   =\langle\mathcal S_\alpha:\alpha\leq\kappa\rangle
$
be a closed system of separable normal measures.  Then there is a
forcing extension $V[G]$, preserving cardinals, cofinalities, and
GCH, in which the reflection ordering of stationary subsets of
$\Reg(\kappa)$ realizes the ground-model well-founded poset
$
   (\mathcal S_\kappa,\triangleleft_{\mathcal S}).
$
\end{theorem}
The proof of the above theorem is based on an Easton support iteration,
and its proof  gives somewhat more information.   Suppose that
$
   \langle X_U:U\in\mathcal S_\lambda\rangle
$
is a separating family for $\mathcal S_\lambda$.  In the extension after the forcing at $\lambda$,
the sets $X_U$ form a maximal antichain modulo $\NS_\lambda$ in
$\Reg(\lambda)$.  Moreover, if
$|\mathcal S_\lambda|<\lambda$ and
$S\subseteq\Reg(\lambda)$ is stationary, then the sets $ \Tr_\lambda(S)\cap\Reg(\lambda)$ and
\[
\bigcup\{X_U: \text{there is }W\in\mathcal S_\lambda
   \text{ such that }
   S\cap X_W\text{ is stationary and }
   W\triangleleft_{\mathcal S}U\}
   \]
   are equal modulo the non-stationary ideal $\NS_\lambda$.
In addition,  if
$S\subseteq\Sing(\lambda)$ is stationary, then
$\Reg(\lambda)\setminus\Tr_\lambda(S)\in\NS_\lambda.
$
Finally, the iteration has the property that no forcing
after stage $\lambda$ adds a new subset of $\lambda$.  Consequently
all statements about subsets, stationarity, and traces at $\lambda$
obtained after the $\lambda$-stage remain true in the final
extension.
We now specialize the step $\lambda$ of Witzany's iteration to the situation needed in the
present paper.

\begin{theorem}\cite{Witzany}
\label{thm:witzanylocal}
Suppose that we are at stage $\lambda$ of Witzany's iteration, and we have a
$\triangleleft_{\mathcal S}$-increasing sequence
$\mathcal S_\lambda =\langle U_\eta:\eta<\rho\rangle$ of measures, where $\rho<\lambda.$
Let
also
  $ \langle E_\eta:\eta<\rho\rangle$
be a separating family.
Then, in the extension after stage $\lambda$, and hence in the final
extension, the $E_\eta$ form a maximal antichain modulo
$\NS_\lambda$ in $\Reg(\lambda)$.
Furthermore,
\begin{itemize}
\item[(1)] If $S\subseteq\lambda$ is stationary,
  $ \Tr_\lambda(S)\cap\Reg(\lambda)
   \equiv_{\NS_\lambda}
   \bigcup_{\mu(S)<\eta<\rho}E_\eta,$
   where $\mu(S) = \min\{\xi<\rho:S\cap E_\xi\text{ is stationary}\}$.

   \item[(2)]
If $S\subseteq\lambda$ and $S\cap\Sing(\lambda)$ is stationary, then
$
   \Reg(\lambda)\setminus\Tr_\lambda(S)\in\NS_\lambda.
$
\end{itemize}
No forcing after stage $\lambda$ adds a subset of $\lambda$.
\end{theorem}

\begin{definition}
\label{def:coherent}
Let $0<\rho<\lambda$.  A  family $ \mathcal E^\lambda=\langle E_\eta^\lambda:\eta<\rho\rangle$ is called a
\emph{coherent  Reflection system of
length $\rho$ at $\lambda$} if:
\begin{enumerate}[(C1)]
\item The sets $E_\eta^\lambda$ are pairwise disjoint stationary subsets of
      $\Reg(\lambda)$ and
      $\Reg(\lambda)\setminus\bigcup_{\eta<\rho}E_\eta^\lambda$ is
      nonstationary.
\item For stationary sets $S\subseteq E_\xi^\lambda$ and
      $T\subseteq E_\eta^\lambda$, we have $S<T$ iff $\xi<\eta.$ Furthermore,
      items (1) and (2) of Theorem \ref{thm:witzanylocal} hold for each $E_\eta^\lambda$.
\item If $0<\eta<\rho$ and $\alpha\in E_\eta^\lambda$, then
$ \mathcal E^\alpha=\langle E_\xi^\alpha: \xi<\eta \rangle$
  forms a coherent  Reflection system of length $\eta$ at
      $\alpha$, where     $E_\xi^\alpha=E_\xi^\lambda\cap\alpha.$

\item If $\alpha\in E_0^\lambda$, then $\Reg(\alpha)$ is non-stationary in
      $\alpha$.
\end{enumerate}
\end{definition}
Fix a sequence $ \mathcal E^\lambda$ as above, and
for $A\subseteq\lambda$ set
\[
  \sigma_{\mathcal E^\lambda}(A)=
  \{\eta<\rho:A\cap E_\eta^\lambda\text{ is stationary in }\lambda\}.
\]
When there is no confusion, we write $\sigma_\lambda(A)=\sigma_{\mathcal E^\lambda}(A)$.
\begin{lemma}
\label{lem:regularsupport}
Suppose $S\subseteq\Reg(\lambda)$ and $1\leq j<\omega$. Then
$
  d_j(S)\cap\lambda\subseteq\Reg(\lambda).
$
In particular $\Tr_\lambda(S)\subseteq\Reg(\lambda)$.
\end{lemma}

\begin{proof}
If $\beta$ is a successor ordinal, then it is isolated in $\tau_0$ and remains isolated in every other
$\tau_j$.  If $\beta$ is a limit ordinal which is not a cardinal, then the cardinals below
$\beta$ are bounded, so $S\cap\beta$ is not  $0$-s-stationary.  Suppose
now that $\beta$ is a singular cardinal and let
$\mu=\operatorname{cf}(\beta)<\beta$.
Let $\langle\beta_i:i<\mu\rangle$ be a continuous and cofinal sequence containing no regular cardinal.  Then $\{\beta_i:i<\mu  \}$ is a club and is disjoint from $S$.  Hence $S\cap\beta$ is non-stationary and cannot
be $j$-s-stationary.  By Theorem \ref{thm:bagaria}(2), $\beta \notin d_1(S)$.  The result follows.
\end{proof}

\begin{lemma}
\label{lem:rankzero}
Suppose $S\subseteq\Reg(\lambda)$ and $j\geq1$. Then
$
  d_j(S)\cap E_0^\lambda=\emptyset.
$
\end{lemma}

\begin{proof}
Suppose $\alpha\in E_0^\lambda$. Then by  clause (C4) above,
$S\cap\alpha\subseteq\Reg(\alpha)$ non-stationary.  It is therefore not
$j$-s-stationary, and hence by  Theorem \ref{thm:bagaria}(2) $\alpha \notin d_j(S)$.
\end{proof}

\begin{lemma}
\label{lem:localsupport}
Suppose $\mathcal E^\lambda$ is coherent,  $A\subseteq\lambda$, and
$0<\eta<\rho$.  There is a club $C\subseteq\lambda$ such that for every
$\alpha\in E_\eta^\lambda\cap C$,
$
  \sigma_{\mathcal E^\alpha}(A\cap\alpha)=\sigma_{\mathcal E^\lambda}(A)\cap\eta,
$
where
$\mathcal E^\alpha=\langle E_\xi^\lambda\cap\alpha:\xi<\eta\rangle$.
\end{lemma}

\begin{proof}
Fix $\xi<\eta$.  If $A\cap E_\xi^\lambda$ is stationary, then by (C2) there is  a club
$D_\xi$ such that every $\alpha\in E_\eta^\lambda\cap C_\xi$ belongs to
$\Tr_\lambda(A\cap E_\xi^\lambda)$.  If it is non-stationary, choose a club
$C_\xi$ disjoint from it,  and let
$D_\xi$ be the set of limit points of $C_\xi$.  Then for $\alpha\in E_\eta^\lambda\cap D_\xi$, the set
$D_\xi\cap\alpha$ is club in $\alpha$, so
$A\cap E_\xi^\lambda\cap\alpha$ is non-stationary. Let $C=\bigcap_{\xi < \eta} D_\xi$. As $\lambda$ is regular and $\eta < \lambda$, $C$ is a club and by (C3) the equality follows immediately.
\end{proof}

We now show that Witzany's model gives us enough coherent Reflection systems.
Let
\[
 \vecU=\langle U(\lambda,\eta):\lambda\leq\kappa,
 \ \eta<o^{\vecU}(\lambda)\rangle
\]
be a coherent sequence of normal measures
Thus, if
$j_{\lambda,\eta}:V\to M_{\lambda,\eta}\simeq \text{Ult}(V, U(\lambda,\eta))$ is the ultrapower embedding by
$U(\lambda,\eta)$, then
$
  o^{j_{\lambda,\eta}(\vecU)}(\lambda)=\eta,
$
and the measures on $\lambda$ seen in that ultrapower are exactly the measures
$U(\lambda,\xi)$, $\xi<\eta$.
Fix a cardinal $\Theta<\kappa$,  for $\Theta<\lambda\leq\kappa,$ set
\[
 \mathcal S^\Theta_\lambda=
 \{U(\lambda,\eta):\eta<\min(o^{\vecU}(\lambda),\lambda)\},
 \qquad \Theta<\lambda\leq\kappa,
\]
and put $\mathcal S^\Theta_\lambda=\emptyset$ for $\lambda\leq\Theta$.
For every $\eta<\kappa$ define the set $X_\eta$ as
\[
 X_\eta=\{\alpha<\kappa:\Theta<\alpha,\ \eta<\alpha,
 \ \alpha\text{ is a regular cardinal, and }
 o^{\vecU}(\alpha)=\eta\}.
\]
The $X_\eta$ are  pairwise disjoint.

\begin{lemma}
\label{lem:separation}
Suppose $\Theta<\lambda\leq\kappa$ and
$
  \xi,\eta<\min(o^{\vecU}(\lambda),\lambda).
$
Then
\[
  X_\xi\cap\lambda\in U(\lambda,\eta)
  \quad\Longleftrightarrow\quad \xi=\eta.
\]
\end{lemma}

\begin{proof}
Let $j=j_{\lambda,\eta}$.  Since $\xi,\Theta<\lambda=\operatorname{crit}(j)$, we have $j(\xi)=\xi$
and $j(\Theta)=\Theta$.
Then,
\[
 X_\xi\cap\lambda\in U(\lambda,\eta)
 \Longleftrightarrow \lambda\in j(X_\xi) \Longleftrightarrow  o^{j(\vecU)}(\lambda)=\xi  \Longleftrightarrow \eta=\xi.
\]
The result follows.
\end{proof}

\begin{lemma}
\label{lem:closedness}
The system $\mathcal S^\Theta=
\langle\mathcal S^\Theta_\lambda:\lambda\leq\kappa\rangle$ is closed  and all of its local measure families are separable.
\end{lemma}

\begin{proof}
Let $W=U(\lambda,\eta)\in\mathcal S^\Theta_\lambda$ and let $j$ be the
corresponding ultrapower embedding.  Since $j(\Theta)=\Theta$ and
$o^{j(\vecU)}(\lambda)=\eta<\lambda$, we have
\[
  (j(\mathcal S^\Theta))_\lambda
  =\{U(\lambda,\xi):\xi<\eta\}
  \subseteq\mathcal S^\Theta_\lambda.
\]
Thus $\mathcal S^\Theta$ is closed.   Separability follows from
Lemma \ref{lem:separation}.
\end{proof}
By applying Theorems \ref{thm:witzany-general} and \ref{thm:witzanylocal} to the system $\mathcal S^\Theta=
\langle\mathcal S^\Theta_\lambda:\lambda\leq\kappa\rangle$, we obtain the following.
\begin{proposition}
\label{prop:cutoff}
Assume GCH holds and  $\vecU$ is a coherent sequence of measures  as above.  There is a cofinality and GCH preserving Easton-support forcing
iteration
\[
\mathbb{P}:=\langle \langle \mathbb{P}_\lambda: \lambda \leq \kappa     \rangle, \langle \dot{\mathbb{Q}}_\lambda: \lambda < \kappa       \rangle\rangle
\]
which has the following properties:
\begin{enumerate}[(1)]
\item $\mathbb{P}_\Theta$ is the trivial forcing,
\item at every stage $\lambda>\Theta$, $\Vdash_{\mathbb{P}_\lambda}$`` $\dot{\mathbb{Q}}_\lambda$  is Witzany's forcing for
      $\mathcal S^\Theta_\lambda$ with separators $X_\eta\cap\lambda$'',
\item For every $\lambda < \kappa$, the forcing $\mathbb{P}_\kappa / \dot{G}_{\mathbb{P}_\lambda}$ adds no new bounded sequences to $\lambda$,
\item if $\lambda\in X_\rho$ and $0<\rho<\lambda$, then in the final extension
      \[
        \langle X_\eta\cap\lambda:\eta<\rho\rangle
      \]
      is a coherent  Reflection system of length $\rho$ at
      $\lambda$, and if $\lambda\in X_0$, then $\Reg(\lambda)$ is non-stationary.
\end{enumerate}
\end{proposition}
Let $\mathbb{P}$ be the forcing notion of Proposition \ref{prop:cutoff}, and let $G$ be $\mathbb{P}$-generic over $V$.
\begin{corollary}
\label{cor:rootsupply}
Assume $o^{\vecU}(\kappa)\geq\kappa$.  Then in $V[G]$, for every $\eta<\kappa$, the set
 $X_\eta$ is stationary in $\kappa$.
Consequently, for every $\eta,\gamma<\kappa$ there is
$\lambda\in X_\eta$ above $\max\{\Theta,\eta,\gamma\}$.
\end{corollary}
\begin{proof}
By our assumption, $ \mathcal S^\Theta_\kappa=
 \{U(\lambda,\eta):\eta<\kappa\}$, and the sequence $\langle X_\eta\cap\kappa: \eta<\kappa  \rangle$
separates it. By Proposition \ref{prop:cutoff}(4), each set $X_\eta\cap\kappa$ is stationary.
The result follows.
\end{proof}

\section{Completeness of $\GL$ for Mahlo topology}
\label{sec:mahlo-case}

In this section, we consider the case of the Mahlo topology and show the completeness of $\GL$  with respect to it.
Let \(\mathcal E^\lambda=\langle E_\eta^\lambda:\eta<\rho\rangle\) be a coherent  Reflection system of length \(0<\rho<\lambda\).
Suppose \(S\subseteq\lambda\) is stationary. If \(S\cap\Sing(\lambda)\) is stationary, put \(\mu(S)=-1\), otherwise set
\[
  \mu(S)=\min\{\xi<\rho:S\cap E_\xi^\lambda\text{ is stationary}\}.
\]

\begin{lemma}\label{lem:mahlo-trace-tail}
For every stationary set \(S\subseteq\lambda\),
$
  \Tr_\lambda(S)\cap\Reg(\lambda)
  \equivNS
  \bigcup_{\mu(S)<\eta<\rho}E_\eta^\lambda.
$
In particular, if \(S,T\subseteq\lambda\) are stationary and
\(\delta=\max\{\mu(S),\mu(T)\}\), then
\[
 \Tr_\lambda(S)\cap\Tr_\lambda(T)\cap\Reg(\lambda)
 \equivNS
 \bigcup_{\delta<\eta<\rho}E_\eta^\lambda.
\]
\end{lemma}

\begin{proof}
If \(S\cap\Sing(\lambda)\) is stationary, then by clause (C2) of Definition \ref{def:coherent}, the set \(S\) reflects at almost every regular cardinal, i.e., $
  \Tr_\lambda(S)\cap\Reg(\lambda)
  \equivNS
  \bigcup_{\eta<\rho}E_\eta^\lambda.
$  This is exactly the first formula of lemma with \(\mu(S)=-1\).

Now suppose that  \(S\cap\Sing(\lambda)\) is not stationary. Then clearly $\Tr_\lambda(S) \equivNS \Tr_\lambda(S\cap\Reg(\lambda))$, so we may replace $S$ by $S\cap\Reg(\lambda).$
  Let \(\delta=\mu(S)\). Again by  clause (C2),
$
  \Tr_\lambda(S)\cap\Reg(\lambda)
  \equivNS
  \bigcup_{\delta<\eta<\rho}E_\eta^\lambda.
$

The second part of the lemma follows immediately by intersecting the resulting equalities. More precisely,   we have
\[
\Tr_\lambda(S)\cap\Tr_\lambda(T)\cap\Reg(\lambda)
 \equivNS
\big( \bigcup_{\mu(S)<\eta<\rho}E_\eta^\lambda \big) \cap \big(  \bigcup_{\mu(T)<\eta<\rho}E_\eta^\lambda        \big).
\]
Now by clause (C1) of Definition \ref{def:coherent}, the set $\Reg(\lambda) \setminus \bigcup_{\eta<\delta}E_\eta^\lambda$
is non-stationary,  hence $\Tr_\lambda(S)\cap\Tr_\lambda(T)\cap\Reg(\lambda)
 \equivNS
 \bigcup_{\delta<\eta<\rho}E_\eta^\lambda$, as requested.
\end{proof}

\begin{lemma}
\label{lem:mahlo-rank-shape}
\begin{enumerate}
\item If \(\rho\) is a nonzero limit ordinal, then \(\lambda\) is \(2\)-s-reflecting and \(\calF^2_\lambda\) is proper.

\item If \(\rho=\delta+1\), then \(\lambda\) is not \(2\)-s-reflecting and hence
$ \calF^2_\lambda=\Pow(\lambda)$ and
$  (\calF^2_\lambda)^+=\emptyset.$
\end{enumerate}
\end{lemma}

\begin{proof}
(1). Suppose \(S,T\subseteq\lambda\) are stationary set, and let  \(\delta > \mu(S), \mu(T)\).  By  Lemma \ref{lem:mahlo-trace-tail},
$ \Tr_\lambda(S)\cap\Tr_\lambda(T)\cap\Reg(\lambda)
 \equivNS
 \bigcup_{\delta<\eta<\rho}E_\eta^\lambda$.
Hence,  $\Tr_\lambda(S)\cap\Tr_\lambda(T) \cap\Reg(\lambda)  \neq \emptyset.$
By Lemma \ref{thm:bagaria},
 \(\lambda\) is \(2\)-s-reflecting.

(2). Suppose \(\rho=\delta+1\).  By Lemma \ref{lem:mahlo-trace-tail},  \( \Tr_\lambda(E_\delta^\lambda)\) is non-stationary on the regular cardinals, and by Lemma \ref{lem:regularsupport} it is a subset of $\Reg(\lambda).$  Let club \(C\) be a club of $\lambda$ which is disjoint to  \( \Tr_\lambda(E_\delta^\lambda)\)
and set \(S=E_\delta^\lambda\cap C\).  Then \(S\) is stationary and \(\Tr_\lambda(S)=\emptyset\).
 Thus the stationary pair \((S,S)\) has no common reflection point, so \(\lambda\) is not \(2\)-s-reflecting. Now  Theorem \ref{thm:bagaria}(5)
 completes the proof.
\end{proof}
\begin{definition}
For \(B\subseteq\rho\), set
$
  D_B^\lambda=\bigcup_{\eta\in B}E_\eta^\lambda.
$
\end{definition}
\begin{theorem}\label{thm:mahlo-tail}
Assume  \(\rho\) is a nonzero limit ordinal and \(B\subseteq\rho\).
\begin{enumerate}
\item $ D_B^\lambda\in\calF^2_\lambda$ iff
  $B$ contains a final segment of $\rho$.

\item  $D_B^\lambda\in(\calF^2_\lambda)^+$
 iff
  $B$  is unbounded in $\rho.$
\end{enumerate}
\end{theorem}

\begin{proof}
(1).
First suppose that \((\delta,\rho)\subseteq B\), for some $\delta < \rho$. By Lemma \ref{lem:mahlo-trace-tail},
 \( \Tr_\lambda(E_\delta^\lambda) \equivNS \bigcup_{\delta<\eta<\rho}E_\eta^\lambda\), hence by Lemma  \ref{lem:NSinvariance},  \(\bigcup_{\delta<\eta<\rho}E_\eta^\lambda \in \calF^2_\lambda\), and therefore  \(D_B^\lambda \in \calF^2_\lambda\).

Conversely suppose that \(D_B^\lambda\in\calF^2_\lambda\).  By  Theorem \ref{thm:bagaria},  \(d_k(S)\cap d_k(T) \subseteq D_B^\lambda\) for some $k<2$
and $k$-s-stationary sets $S, T$.
  We show that $k=1$. Otherwise $k=0$. Hence the sets \(S,T\)  are unbounded, and  $d_k(S)\cap d_k(T) \subseteq D_B^\lambda \subseteq\Reg(\lambda)$ is  a club. But \(\Sing(\lambda)\) is stationary and has empty intersection with $d_k(S)\cap d_k(T)$, a contradiction.  Thus  \(S,T\) are stationary.  Let \(\delta=\max\{\mu(S),\mu(T)\}\).  By Lemma \ref{lem:mahlo-trace-tail},
 \[
 d_k(S)\cap d_k(T) \cap \Reg(\lambda) = \Tr_\lambda(S) \cap \Tr_\lambda(T) \cap \Reg(\lambda) \equivNS \bigcup_{\delta<\eta<\rho}E_\eta^\lambda.
 \]
  But \(d_k(S)\cap d_k(T) \subseteq D_B^\lambda\), and hence $(\delta, \rho) \subseteq B.$

(2).
First suppose that \(B\) is unbounded in $\rho$ and let \(A\in\calF^2_\lambda\).
 As above,  \(d_k(S)\cap d_k(T) \subseteq A\) for some $k<2$
and $k$-s-stationary sets $S, T$.
 If $k=0$, then $d_k(S)\cap d_k(T)$ is a club, hence $A$ contains  a club and therefore   \(A \cap E_\eta^\lambda \neq \emptyset\) for some  \(\eta\in B\). In particular, $A \cap D_B^\lambda \neq \emptyset$  If $k=1$,  the sets $S$ and $T$ are stationary.  Choose \(\eta\in B\) above $\max\{\mu(S),\mu(T)\}$. By Lemma \ref{lem:mahlo-trace-tail}  \(d_k(S)\cap d_k(T) \cap E_\eta^\lambda \neq \emptyset\).  Thus \(A\cap D_B^\lambda\neq\emptyset\).

If \(B\) is bounded, choose \(\delta<\rho\) with \(B\subseteq\delta+1\) and \(\delta+1<\rho\).
By (1), $D_{(\delta+1, \rho)}^\lambda \in \calF^2_\lambda$
 and is disjoint from \(D_B^\lambda\).  Thus \(D_B^\lambda\) is not positive.
\end{proof}

\begin{theorem}\label{thm:mahlo-single-cut}
Let \(\langle B_s:s\in K_n\rangle\) be an \((\calF^0, \rho)\)-realization, where $\rho >0$, and suppose \(\lambda\) carries a coherent  Reflection system \(\langle E_\eta^\lambda:\eta<\rho\rangle\).  Define $\Gamma: K_n \rightarrow \mathcal{P}(\lambda+1)$ by
\[
 \Gamma(\emptyset)=\{\lambda\},\qquad
 \Gamma(s)=\bigcup_{\eta\in B_s}E_\eta^\lambda\quad(s\neq\emptyset).
\]
Then \(\Gamma\) satisfies the hypotheses of Theorem \ref{thm:blasscriterion} for the Mahlo filters \(\calF^2\).
\end{theorem}

\begin{proof}
We show that  \(\Gamma\) satisfies items (1)-(4) of Theorem \ref{thm:blasscriterion}. Clause (1) is clear. For (2), if $s \neq t$, then we have
$B_s \cap B_t=\emptyset$ and the sets $E_\eta^\lambda$ are also pairwise disjoint, hence $\Gamma(s) \cap \Gamma(t)=\emptyset.$
For (3), let  \(\alpha\in\Gamma(s)\) and \(s\triangleleft t\).  If $s=\emptyset$, set \(r(\alpha)=\rho\),  otherwise set \(r(\alpha)=\eta\), where \(\eta\in B_s\) is the unique ordinal  with \(\alpha\in E_\eta^\lambda\).

 By the choice of the sequence \(\langle B_s:s\in K_h\rangle\),  the set \(B_t\cap r(\alpha)\) is unbounded in \(r(\alpha)\).  Thus \(r(\alpha)\) is a nonzero limit ordinal and hence by Theorem \ref{thm:mahlo-tail},
\[
 \Gamma(t)\cap\alpha
 =\bigcup_{\xi\in B_t\cap r(\alpha)}(E_\xi^\lambda\cap\alpha)
 \in(\calF^2_\alpha)^+.
\]
For (4), if \(\alpha\in\Gamma(s)\) and $s$ is not a terminal node,  then by arguments as above,  the  set $ \bigcup_{s\triangleleft t}(\Gamma(t)\cap\alpha)$ contains a final segment of \(r(\alpha)\), and hence it belongs to \(\calF^2_\alpha\). If $s$ is a terminal node, then \(r(\alpha)\) has to be a successor rank, and Lemma \ref{lem:mahlo-rank-shape} implies \(\calF^2_\alpha\) is improper, so the empty  union is allowed.
\end{proof}

\begin{corollary}\label{cor:mahlo-abstract}
Suppose for every finite $n$, and an   \((\calF^0, \rho_n)\)-realization with $\rho_n>0$, there is some cardinal   \(\lambda_n\) which carries a coherent  Reflection system \(\langle E_\eta^\lambda:\eta<\rho_n\rangle\).  Then for every ordinal $\alpha \geq \sup_{n<\omega}\lambda_n$,
 \(\GL\) is complete with respect to the ordinal space  $(\alpha, \tau_2)$, where $\tau_2$ is the Mahlo topology.
\end{corollary}
\begin{corollary}\label{cor:mahlo-completeness1}
Assume GCH holds and there is a measurable cardinal $\kappa$ which carries a  coherent sequence $\vecU$
with $o^{\vecU}(\kappa)\geq\kappa$. Then there is a generic extension $V[G]$ of the universe in which \(\GL\) is complete with respect to the ordinal space  $(\alpha, \tau_2)$, for every $\alpha \geq \kappa.$
\end{corollary}
\begin{proof}
By Theorem \ref{thm:tailbase}, Proposition \ref{prop:cutoff} and Corollary \ref{cor:mahlo-abstract}.
\end{proof}

\section{The transfer theorem}
\label{sec:shift}
In the previous section, we showed how to derive completeness of $\GL$ for the Mahlo topology from its completeness for
the interval topology,  using a coherent  Reflection system. We show that this extends to higher derivatives.
\begin{theorem}
\label{thm:shift}
Let $\mathcal E^\lambda=\langle E_\eta^\lambda:\eta<\rho\rangle$ be a
coherent  Reflection system and set  $ \sigma_{\lambda}= \sigma_{\mathcal E^\lambda}$.  Then for $m<\omega$ and
$A\subseteq\lambda$,
\[
 A\text{ is }(m+2)\text{-s-stationary in }\lambda
 \quad\Longleftrightarrow\quad
 \sigma_\lambda(A)\text{ is }m\text{-s-stationary in }\rho.
 \tag{$\text{Transfer}_m$}
\]
Moreover, if $B=\sigma_\lambda(A)$ and $0<\eta<\rho$, then
\begin{enumerate}
\item[$(1)_m$] $ \eta\in d_m(B)\Longrightarrow
 E_\eta^\lambda\setminus d_{m+2}(A)\in\NS_\lambda,$
\item[$(2)_m$] $\eta\notin d_m(B)\Longrightarrow
 E_\eta^\lambda\cap d_{m+2}(A)\in\NS_\lambda.$
\end{enumerate}
\end{theorem}

\begin{proof}
We prove the theorem by
induction on $m$, and  over all coherent systems. In what follows we repeatedly refer to items (C1)-(C4) of Definition \ref{def:coherent}.

First suppose that $m=0$.
Set $B=\sigma_\lambda(A)=\{ \eta < \rho: A \cap E^\lambda_\eta$ is stationary   $ \}$.  Suppose first that $B$ is unbounded in $\rho$.
Then for
some $\eta\in B$, the set $A\cap E_\eta^\lambda$ is stationary, and hence $A$ is stationary.
If $S,T\subseteq\lambda$ are unbounded, then $d_0(S)$ and $d_0(T)$ are
clubs in  $\lambda$, and hence
$
  A\cap d_0(S)\cap d_0(T)\neq\emptyset.
$
Now let $S,T$ be stationary. We show that $ A\cap\Tr_\lambda(S)\cap\Tr_\lambda(T)\neq\emptyset.$ There are several case to consider:
\begin{itemize}
\item The sets $S\cap\Sing(\lambda)$ and $T\cap\Sing(\lambda)$ are stationary.  Then
by (C2),
 $ \Reg(\lambda)\setminus\Tr_\lambda(S)\in\NS_\lambda,$ and $ \Reg(\lambda)\setminus\Tr_\lambda(T)\in\NS_\lambda.$
 Hence $\Tr_\lambda(S)\cap\Tr_\lambda(T) \supseteq_{\NS} \Reg(\lambda)$.
Pick $\eta < \rho$ such that $A \cap E^\lambda_\eta$ is stationary.  Then $A\cap\Tr_\lambda(S)\cap\Tr_\lambda(T) \supseteq_{\NS}  A \cap \Reg(\lambda) \supseteq A \cap E^\lambda_\eta$, thus $ A\cap\Tr_\lambda(S)\cap\Tr_\lambda(T)\neq\emptyset.$

\item The sets $S\cap\Reg(\lambda)$ and $T\cap\Reg(\lambda)$ are stationary. By  (C2),
   $ \Tr_\lambda(S) \cap\Reg(\lambda)
  \equivNS\bigcup_{\mu(S)<\eta<\rho}E^\lambda_\eta$, and
 $ \Tr_\lambda(T) \cap\Reg(\lambda)
  \equivNS\bigcup_{\mu(T)<\eta<\rho}E^\lambda_\eta$. By (C1),
  $\Reg(\lambda) \setminus \bigcup_{\eta<\rho}E^\lambda_\eta$ is non-stationary. Let $\mu=\max\{\mu(S), \mu(T)\}$ and pick $\mu < \delta \in B$.
  Then $ A\cap\Tr_\lambda(S)\cap\Tr_\lambda(T) \equivNS\big(\bigcup_{\mu<\eta<\rho}E^\lambda_\eta\big) \cap A \supseteq A \cap E^\lambda_\delta$,
  hence $ A\cap\Tr_\lambda(S)\cap\Tr_\lambda(T)\neq\emptyset.$

\item  The set $S\cap\Sing(\lambda)$ is stationary and the set $T\cap\Reg(\lambda)$  is stationary. Then as above,
 $ \Reg(\lambda)\setminus\Tr_\lambda(S)\in\NS_\lambda,$ and  $ \Tr_\lambda(T) \cap\Reg(\lambda)
  \equivNS\bigcup_{\mu(T)<\eta<\rho}E^\lambda_\eta$. Thus if $\mu(T) <\delta \in B$, then
$A \cap \Tr_\lambda(S) \cap   \Tr_\lambda(T)
  \equivNS \big(\bigcup_{\mu(T)<\eta<\rho}E^\lambda_\eta\big) \cap A \supseteq A \cap E^\lambda_\delta$,
  and the result follows again.

\item  The set $T\cap\Sing(\lambda)$ is stationary and the set $S\cap\Reg(\lambda)$  is stationary.
The proof is the same as above.
\end{itemize}
Thus by Theorem \ref{thm:bagaria}(3),  $A$ is
$2$-s-stationary.

Conversely, suppose $B$ is bounded in $\rho$.  Let $\eta<\rho$ be such that
 $B \subseteq \eta$ and set $S=E_\eta^\lambda$.  By (C2), $\Tr_\lambda(S)  \equivNS\bigcup_{\eta<\mu<\rho}E^\lambda_\mu$, and by
Lemma \ref{lem:regularsupport},  $\Tr_\lambda(S) \subseteq \Reg(\lambda)$.
By the choice of $\eta$, the set $A \cap E^\lambda_\mu$ is non-stationary for every $\eta<\mu<\rho$,
and hence by $\lambda$-completeness of $\NS_\lambda$, the set $A \cap \big(\bigcup_{\eta<\mu<\rho}E^\lambda_\mu\big)$ is non-stationary.
Thus by (C1),
$
  A\cap\Tr_\lambda(S)\in\NS_\lambda.
$
Let $C$ be a club of $\lambda$ disjoint from $ A\cap\Tr_\lambda(S)$ and set $T=S\cap C$.
Then $T$ is stationary, and for $\beta\in\Tr_\lambda(T)$, the set $T\cap\beta$
is unbounded in $\beta$, and hence  $\beta\in C$.  It follows that
$
  \Tr_\lambda(S')\subseteq\Tr_\lambda(S)\cap C,$ and
$A\cap\Tr_\lambda(S')=\emptyset.$
Thus the pair $(T, T)$ witnesses that $A$ is not $2$-s-stationary. This proves $\text{Transfer}_0$.

Next fix $0<\eta<\rho$.  By Lemma \ref{lem:localsupport}, there is a club $C$, such that for all
$\alpha\in C \cap E_\eta^\lambda$,
$
  \sigma_\alpha(A\cap\alpha)=B\cap\eta.
$
Applying $\text{Transfer}_0$ to the local system at $\alpha$, we get
\[
 \alpha\in d_2(A)
 \Longleftrightarrow
 B\cap\eta\text{ is unbounded in }\eta
 \Longleftrightarrow
 \eta\in d_0(B),
\]
which gives items $(1)_0$ and $(2)_0$.

Now assume that  $\text{Transfer}_k$ and items $(1)_k$ and $(2)_k$ hold
 for every $k\leq m$ and all coherent systems.  We  prove $\text{Transfer}_{m+1}$.
First suppose that $B=\sigma_\lambda(A)$ is $(m+1)$-s-stationary in $\rho$.  It is unbounded, so
by $\text{Transfer}_0$, the set $A$ is $2$-s-stationary. We have to show that $A$ is indeed $(m+3)$-s-stationarity.
Fix $k\leq m$ and let $S,T\subseteq\lambda$ be
$(k+2)$-s-stationary.  By $\text{Transfer}_k$, the sets
$U=\sigma_\lambda(S)$ and $ V=\sigma_\lambda(T)$
are $k$-s-stationary in $\rho$.  Since $B$ is
$(m+1)$-s-stationary, by Theorem \ref{thm:bagaria}(3),
$\eta\in B\cap d_k(U)\cap d_k(V)$
for some $\eta<\rho$.  Clearly $\eta>0$, as $0$ is isolated in
every $\tau_k$.
By $(1)_k$, $E_\eta^\lambda \setminus d_{k+2}(S) \in \NS_\lambda$
and $E_\eta^\lambda \setminus d_{k+2}(T) \in \NS_\lambda$.
  Since $A\cap E_\eta^\lambda$ is stationary,
$
  A\cap d_{k+2}(S)\cap d_{k+2}(T)\neq\emptyset.
$
Thus by Theorem \ref{thm:bagaria}(3), $A$ is $(m+3)$-s-stationary.

For the converse, suppose $B$ is not $(m+1)$-s-stationary.  By
Theorem \ref{thm:bagaria}(3), there are $k\leq m$ and $k$-s-stationary sets
$U,V\subseteq\rho$ such that
$
  B\cap d_k(U)\cap d_k(V)=\emptyset.
$
Set
$
  S=\bigcup_{\xi\in U}E_\xi^\lambda,$
and
$  T=\bigcup_{\xi\in V}E_\xi^\lambda.$
By  $\text{Transfer}_k$, the sets $S$ and $T$ are
$(k+2)$-s-stationary.  Let
$
  H=d_{k+2}(S)\cap d_{k+2}(T).
$
If $0<\eta\in B$, then  $0<\eta \notin d_k(U)\cap d_k(V)$, thus by $(2)_k$, i
$H\cap E_\eta^\lambda$ is non-stationary.  If $\eta\notin B$, then by its definition, $A\cap E_\eta^\lambda$ is
non-stationary.  Thus
$
 A\cap H\cap\big(\bigcup_{0<\eta<\rho}E_\eta^\lambda\big)
 \in\NS_\lambda.
$
On the other hand,  by Lemma \ref{lem:rankzero},
$H\cap E_0^\lambda=\emptyset$, by Lemma \ref{lem:regularsupport},
$H\subseteq\Reg(\lambda)$, and by (C1)  $\Reg(\lambda) \setminus \bigcup_{0<\eta<\rho}E_\eta^\lambda \in \NS_\lambda$. Hence $A\cap H \in \NS_\lambda$.

Let $C$ be a club of $\lambda$ disjoint from $A\cap H$, and set
$S'=S\cap C$ and $T'=T\cap C$.  Then $S'$ and $T'$ are
$(k+2)$-s-stationarity,
$d_{k+2}(S')\subseteq d_{k+2}(S)$ and $d_{k+2}(T')\subseteq d_{k+2}(T)$.  Moreover, every
point of either derivative belongs to $C$, because the corresponding initial
segment of $S'$ or $T'$ is unbounded and $C$ is closed.  Therefore
\[
 A\cap d_{k+2}(S')\cap d_{k+2}(T')
 \subseteq A\cap H\cap C=\emptyset.
\]
This witnesses the failure of $(m+3)$-s-stationarity and completes
$(\text{Transfer}_{m+1})$.

Finally fix $0<\eta<\rho$. We show items $(1)_{m+1}$ and $(2)_{m+1}$ hold.
 By Lemma \ref{lem:localsupport}, there is a club $C$, such that for all
$\alpha\in C \cap E_\eta^\lambda$,
$\sigma_\alpha(A\cap\alpha)=B\cap\eta$.  Thus by $(\text{Transfer}_{m+1})$,
\begin{align*}
 \alpha\in d_{m+3}(A)
 &\Longleftrightarrow A\cap\alpha
     \text{ is }(m+3)\text{-s-stationary in }\alpha\\
 &\Longleftrightarrow B\cap\eta
     \text{ is }(m+1)\text{-s-stationary in }\eta\\
 &\Longleftrightarrow \eta\in d_{m+1}(B).
\end{align*}
This proves $(1)_{m+1}$ and $(2)_{m+1}$. The theorem follows.
\end{proof}

Let us recall that for$B\subseteq\rho$, the set $D^\lambda_B=$ was defined as
$
  D^\lambda_B=\bigcup_{\eta\in B}E_\eta^\lambda.
$
We have the following immediate corollary.
\begin{corollary}
\label{cor:filtertransfer}
Assume the hypotheses of Theorem \ref{thm:shift} holds, $m<\omega$ and
$B\subseteq\rho$. Then
\begin{enumerate}
\item $ D^\lambda_B=\in\calF^{m+2}_\lambda
 \Longleftrightarrow B\in\calF^m_\rho$.

\item $ D^\lambda_B=\in(\calF^{m+2}_\lambda)^+
 \Longleftrightarrow B\in(\calF^m_\rho)^+.$
\end{enumerate}
\end{corollary}
\section{Completeness of  $\GL$ for higher derived topologies}
\label{sec:lifting}
The proof of the next theorem is very similar to the proof of Theorem \ref{thm:mahlo-single-cut}.
\begin{theorem}
\label{thm:labellift}
Let $m,n<\omega$.  Suppose
$\langle B_s:s\in K_n\rangle$ is an $(\calF^m, \rho)$-realization, where
$0<\rho$, and $\lambda$ carries a coherent system
$\langle E_\eta^\lambda:\eta<\rho\rangle$.  Define $\Gamma: K_n \rightarrow \mathcal{P}(\lambda+1)$ by
\[
  \Gamma(\emptyset)=\{\lambda\},
  \qquad
  \Gamma(s)=D_\lambda(B_s)=\bigcup_{\eta\in B_s}E_\eta^\lambda
  \quad(s\neq\emptyset).
\]
Then $\Gamma$ is an $(\calF^{m+2}, \lambda)$-realization.
\end{theorem}

\begin{corollary}
\label{cor:abstractcomplete}
Fix $m<\omega$. Suppose for every finite $n$, there is an   \((\calF^m, \rho_n)\)-realization with $\rho_n>0$, and there is some cardinal   \(\lambda_n\) which carries a coherent  Reflection system \(\langle E_\eta^\lambda:\eta<\rho_n\rangle\).  Then for every ordinal $\alpha \geq \sup_{n<\omega}\lambda_n$,
 \(\GL\) is complete with respect to the ordinal space  $(\alpha. \tau_{m+2})$.
\end{corollary}

Putting the results together, we obtain the main result of the paper.
\begin{theorem}
\label{cor:concrete}
Assume GCH holds,  and   $\kappa$ carries a coherent sequence $\vecU$
with $o^{\vecU}(\kappa)\geq\kappa$. Then there is a cofinality and GCH preserving generic extension $V[G]$ of $V$
in which $\GL$ is complete for each filter sequence
$\langle\calF^m_\alpha:\alpha<\kappa\rangle$, and hence
$\Logd(\kappa,\tau_m)=\GL$ for every $m<\omega$
\end{theorem}
\begin{proof}
By a preliminary small forcing, or going to the inner moder, we can assume that $\square_\mu$ holds for every infinite cardinal
$\mu<\aleph_\omega$.
Choose a cardinal $\Theta$ with $\aleph_\omega<\Theta<\kappa$ and apply
Proposition~\ref{prop:cutoff}. Since the iteration is trivial at and below
$\Theta$ and adds no subset of $\Theta$, the two families of base
realizations supplied by Theorems~\ref{thm:tailbase} and
\ref{thm:clubbase} remain available in the final extension.

We construct, for every $m,n<\omega$, an
$(\calF^m,\rho^m_n)$-realization of $K_n$ with
$0<\rho^m_n<\kappa$. For $m=0$, use Theorem~\ref{thm:tailbase}; for $m=1$,
use Theorem~\ref{thm:clubbase}. Suppose an
$(\calF^m,\rho^m_n)$-realization has been constructed. By
Corollary~\ref{cor:rootsupply}, choose
$
 \lambda^m_n\in X_{\rho^m_n}$
such that
$ \lambda^m_n>\max\{\Theta,\rho^m_n\}.
$
Proposition~\ref{prop:cutoff}(4) gives at $\lambda^m_n$ a coherent
reflection system of length $\rho^m_n$. Theorem~\ref{thm:labellift} then
lifts the given realization to an
$(\calF^{m+2},\lambda^m_n)$-realization. Put
$\rho^{m+2}_n=\lambda^m_n$. This recursion constructs all even levels from
the interval base and, using square, all odd levels from the club base.

Fix $m<\omega$. For each $n<\omega$, regard the resulting rooted
realization as a labeling by subsets of $\kappa$. Since its root is below
$\kappa$, it satisfies the hypotheses of
Theorem~\ref{thm:blasscriterion} for
$\langle\calF^m_\alpha:\alpha<\kappa\rangle$. Hence every formula valid for
that filter sequence belongs to $\GL$. The reverse inclusion, and the
identification with derivative semantics on $(\kappa,\tau_m)$, follow from
Corollary~\ref{cor:soundness}.
\end{proof}

\section{On $\GLP$}
\label{sec:polymodal-lifting}

In this section we prove a consistency result for $\GLP_2$, that extends the earlier work of Beklemishev \cite{BeklemishevGLB}.
For a formula $\varphi$ in the modalities $[0],\ldots,[r-1]$, let
$\varphi^{\uparrow2}$ be obtained by replacing each $[i]$ by $[i+2]$.
Let $\mathcal E^\lambda=\langle E_\eta^\lambda:\eta<\rho\rangle$ be a
coherent Reflection system as in Definition~\ref{def:coherent}.  
A set
$A\subseteq\lambda$ is called \emph{$\mathcal E^\lambda$-saturated} if,
for every $0<\eta<\rho$, either
$
 A\cap E_\eta^\lambda\in\NS_\lambda$
or
$ E_\eta^\lambda\setminus A\in\NS_\lambda.
$
For a saturated set $A$, set
\[
 s_\lambda(A)=
 \{\eta\in\rho: 0< \eta \text{ and }
 E_\eta^\lambda\setminus A\in\NS_\lambda\}.
\]

\begin{lemma}
\label{lem:poly-saturation}
The $\mathcal E^\lambda$-saturated subsets of $\lambda$ are closed under
 Boolean operations.  Suppose that $A$ is saturated and for some set  $B\subseteq\rho$,
$s_\lambda(A)=B\cap(\rho\setminus\{0\})$.  Then, for every $m<\omega$, $d_{m+2}(A)$ is also
saturated and
$
 s_\lambda(d_{m+2}(A))
 =d_m(B)\cap(\rho\setminus\{0\}).
$
\end{lemma}

\begin{proof}
It is clear that $\mathcal E^\lambda$-saturated sets are closed under  Boolean operations. 
 Moreover the assumption $s_\lambda(A)=B\cap(\rho\setminus\{0\})$ implies 
$
s_\lambda(A)\mathbin\triangle B\subseteq\{0\},
$
and hence
$
 d_m(s_\lambda(A))=d_m(B).
$
By  Theorem~\ref{thm:shift}, for every $0<\eta < \rho$, we have
\begin{itemize}
\item  $\eta\in d_m(B)\Longrightarrow
 E_\eta^\lambda\setminus d_{m+2}(A)\in\NS_\lambda \Longrightarrow \eta \in s_\lambda(d_{m+2}(A)),$
\item $\eta\notin d_m(B)\Longrightarrow
 E_\eta^\lambda\cap d_{m+2}(A)\in\NS_\lambda \Longrightarrow \eta \notin s_\lambda(d_{m+2}(A)).$
\end{itemize}
The result follows immediately.
\end{proof}

 If $v$ is a valuation on $\rho+1$, define  $\widehat v$ on $\lambda+1$ by
\begin{align*}
 \widehat v(p)\cap\lambda
   &=D^\lambda_{v(p)\cap\rho}
     =\bigcup_{\eta\in v(p)\cap\rho}E_\eta^\lambda,\\
 \lambda\in\widehat v(p)&\quad\Longleftrightarrow\quad
 \rho\in v(p).
\end{align*}
Then we have the following easy result.
\begin{theorem}
\label{thm:polymodal-two-step}
Suppose $r\geq1$ and  $v$ is a valuation on
$(\rho+1;\tau_0,\ldots,\tau_{r-1})$.  For every formula $\varphi$ using
the modalities $[0],\ldots,[r-1]$, we have:
\begin{enumerate}[(1)]
\item $\llbracket\varphi^{\uparrow2}\rrbracket_{\widehat v}\cap\lambda$
      is $\mathcal E^\lambda$-saturated;
\item $ s_\lambda(\llbracket\varphi^{\uparrow2}\rrbracket_{\widehat v}
                 \cap\lambda)
 =\llbracket\varphi\rrbracket_v\cap
       (\rho\setminus\{0\});$
\item
$
 (\lambda+1;\tau_2,\ldots,\tau_{r+1}),\lambda
       \models_{\widehat v}\varphi^{\uparrow2}
 \quad\Longleftrightarrow\quad
 (\rho+1;\tau_0,\ldots,\tau_{r-1}),\rho
       \models_v\varphi.
$
\end{enumerate}
\end{theorem}

\begin{proof}
We prove the theorem  by induction on the complexity of $\varphi$. The only non-trivial case is when
$\varphi=\langle m\rangle\psi$.  Set
$
 A=\llbracket\psi^{\uparrow2}\rrbracket_{\widehat v}\cap\lambda,
$
and 
$ B=\llbracket\psi\rrbracket_v\cap\rho.$
By the induction hypothesis, the set $A$ is saturated and
$s_\lambda(A)=B\cap(\rho\setminus\{0\})$. Thus by Lemma
\ref{lem:poly-saturation},
$
 s_\lambda(d_{m+2}(A))
 =d_m(B)\cap(\rho\setminus\{0\}),
$
which proves (1) and (2) for $\varphi$.  For clause (3),
by Theorem~\ref{thm:shift} we have
\begin{align*}
 \lambda\in d_{m+2}(A)
 &\Longleftrightarrow
 A\text{ is }(m+2)\text{-s-stationary in }\lambda\\
 &\Longleftrightarrow
 \sigma_\lambda(A)\text{ is }m\text{-s-stationary in }\rho\\
 &\Longleftrightarrow
 B\text{ is }m\text{-s-stationary in }\rho\\
 &\Longleftrightarrow
 \rho\in d_m(B).
\end{align*}
  This completes the proof.
\end{proof}

Using Beklemishev's  theorem \ref{beklem} and repeated
applications of  Proposition \ref{prop:cutoff}, we get the following theorem.

\begin{theorem}
\label{thm:even-pair-glp2}
Let $V[G]$ be the generic extension obtained  in Theorem~\ref{cor:concrete}.  Then, for every $k<\omega$,
\[
 \Logp(\kappa;\tau_{2k},\tau_{2k+1})=\GLP_2,
\]
where the two modalities of $\GLP_2$ are interpreted respectively by
$d_{2k}$ and $d_{2k+1}$.
\end{theorem}

{\bf Acknowledgements}. The author's research has been supported by a grant from IPM (No. 1405030417). He thanks Reihane Zoghifard for introducing him the problems and useful discussions about derived topologies.

\end{document}